\documentclass{amsart}
\usepackage{amsmath,amssymb}
\usepackage{graphicx}		  

\title{Non-orientable genus of knots in punctured Spin 4-manifolds}
\author{Kouki Sato}
\date{}

\newtheorem{problem}{Problem}
\newtheorem{thm}{Theorem}[section]
\newtheorem{prop}[thm]{Proposition}
\newtheorem{lem}[thm]{Lemma}
\newtheorem{claim}{Claim}
\theoremstyle{definition}
\newtheorem*{dfn}{Definition}
\newtheorem*{acknowledgements}{Acknowledgements}

\begin{document}

\maketitle

\begin{abstract}
For a closed 4-manifold $X$ and a knot $K$ 
in the boundary of punctured $X$, 
we define $\gamma_X^0(K)$ to be the smallest
first Betti number of 
non-orientable and null-homologous surfaces in punctured $X$
with boundary $K$.
Note that $\gamma^0_{S^4}$ is equal to
the non-orientable 4-ball genus and hence $\gamma^0_X$ is a generalization of
the non-orientable 4-ball genus.
While it is very likely that for given $X$, $\gamma^0_X$ has no upper bound,
it is difficult to show it.
In fact, even in the case of $\gamma^0_{S^4}$,
its non-boundedness  was shown for the first time by Batson in 2012.
In this paper, we prove that for any Spin 4-manifold $X$,
$\gamma^0_X$ has no upper bound.
\end{abstract}

\section{Introduction}
Throughout this paper, we work in the smooth category, all 4-manifolds
are orientable, oriented and simply-connected, and all surfaces are compact. 
If $X$ is a closed 4-manifold, punc$X$ denotes $X$ with an open 4-ball
deleted.

In \cite{s-t},
we defined the non-orientable genera of knots
in punc$X$ as follows.
\begin{dfn}
Let $X$ be a closed 4-manifold and
$K \subset \partial (\text{punc}X)\  (\cong S^3)$ a knot.
\textit{The non-orientable $X$-genus  $\gamma_X (K)$ of $K$}
is the smallest first Betti number of any non-orientable
surface $F \subset \text{punc}X$ with boundary $K$.
Moreover, we define $\gamma^0_{X}(K)$ 
to be 
the smallest first Betti number of any non-orientable
surface $F \subset \text{punc}X$ with boundary $K$
which  represents zero in 
$H_2 (\text{punc}X, \partial(\text{punc}X); \mathbb{Z}_2 )$.
\end{dfn}

We note that both $\gamma_{X}(K)$ and $\gamma^0_{X}(K)$
are generalizations
of the non-orientable 4-ball genus $\gamma_4(K)$,
which is the smallest first Betti number of any non-orientable
surface in $B^4$ with boundary $K$.
In this paper, we investigate the following problem.
\begin{problem}
\label{problem}
For a given 4-manifold $X$, do $\gamma_{X}$ and $\gamma^0_{X}$ have upper bounds?
\end{problem}
While $\gamma_4$ has been investigated since 1975 \cite{viro},
it is difficult to evaluate $\gamma_4$
and Problem \ref{problem} had remained open even in the case of $\gamma_4$
until recently.
The best reference for related studies is \cite{gilmer-livingston}.
In 2012, Batson \cite{batson} gave a negative answer of Problem \ref{problem}
in the case of $\gamma_4$
by using Heegaard Floer homology theory.
(In fact, Batson proved that for any positive integer $k$, there exists a knot $K$
which satisfies $\gamma_4(K)=k$.)
In 2014, Tange and the author \cite{s-t} applied Batson's idea to the case of 
$\gamma^0_{n\mathbb{C}P^2}$
and proved that $\gamma^0_{n\mathbb{C}P^2}$ also has no upper bound.
On the other hand,
Suzuki \cite{suzuki} and Norman \cite{norman}
proved that if $X$ is diffeomorphic to $S^2 \times S^2$
or $\mathbb{C}P^2 \# \overline{\mathbb{C}P^2}$, then
any knot bounds a disk in punc$X$.
This result implies that for many 4-manifolds
such as $(S^2 \times S^2) \# Y$ and  $\mathbb{C}P^2 \# \overline{\mathbb{C}P^2} \# Y$,
$\gamma_X$ is bounded above by $1$, where $Y$ denotes any 4-manifold.
Then, it seems natural to consider for a manifold $X$ with $\gamma_{X}$ bounded,
whether or not $\gamma^0_{X}$ also has upper bound. 
In this paper, we prove the following theorem
that gives a negative answer of this question.
\begin{thm}
\label{main thm}
If $X$ is a Spin $4$-manifold, then $\gamma^0_{X}$ has no upper bound.
\end{thm}
By this theorem, it follows that for any Spin $4$-manifold $X$,
$\gamma_{(S^2 \times S^2) \# X}$ is bounded above by $1$
and $\gamma^0_{(S^2 \times S^2) \# X}$ has no upper bound.
Moreover, 
this theorem gives an alternate proof of non-boundedness of $\gamma_4$ 
which is obtained from Furuta's $10/8$-theorem \cite{furuta}
instead of Heegaard Floer homology theory. 

\begin{acknowledgements}
The author would like to thank Kokoro Tanaka and Akira Yasuhara 
for their useful comments and encouragement.
\end{acknowledgements}

\section{Construction of a lower bound for $\gamma^0_{X}$}

Let $Y$ be a closed 4-manifold.
For any knot 
$K \subset \partial(\text{punc}Y)$,
we define Char$(Y, K)$ a set of characteristic homology classes 
in $H_2(\text{punc}Y,\partial(\text{punc}Y); \mathbb{Z})$
which  are represented by disks in punc$Y$ with boundary $K$.
In this section, we construct a new lower bound of $\gamma^0_X$
for any Spin 4-manifold $X$ which consists of the knot signature, 
an element of Char$(Y,K)$  and invariants for $X$.

\begin{prop}
\label{lower bound}
Let $X$ be a closed Spin $4$-manifold and $K$ a knot in $\partial (\text{punc}X)$.
Then for any closed 4-manifold $Y$ and any element $\eta$ of Char$(Y,K)$, we have
\begin{eqnarray*}
3\gamma^0_X(K) &\geq& \left| \sigma(K) - \frac{\eta \cdot \eta}{2} + 
\frac{\sigma(X)+\sigma(Y)}{2} \right| \\
\ &\ & - 4\max(\beta_2^+(X)+\beta_2^-(Y), \beta_2^-(X)+\beta_2^+(Y))-\beta_2(X),
\end{eqnarray*}
where $\sigma(K)$ and $\sigma(X)$ is the signature of $K$ and $X$ respectively, 
$\eta \cdot \eta$ is the self-intersection number of $\eta$, and
$\beta^+_2$ (resp. $\beta^-_2$) is the rank of positive (resp. negative)
part of the intersection form of $X$.
\end{prop}
In order to prove Proposition \ref{lower bound}, 
we need the following lemma.

\begin{lem}
\label{non-orientable}
Let $X$ be a closed $4$-manifold,
$K$ a knot in the boundary of punc$X$,
and $N$ a non-orientable surface in punc$X$ with $\partial N = K$
which represents a characteristic homology class in 
$H_2(\text{punc}X, \partial(\text{punc}X); \mathbb{Z}_2)$.
Then for any closed 4-manifold $Y$ and any element $\eta$ of Char$(Y,K)$, we have
$$
\beta_1(N) \geq \frac{| e(N) - \eta \cdot \eta - \sigma(X)+ \sigma(Y) |}{4} - 
2\max(\beta_2^+(X)+\beta_2^-(Y), \beta_2^-(X)+\beta_2^+(Y)).
$$
\end{lem}

\begin{proof}
We first prove the lemma in the case when $\beta_1(N)$ is odd.
In this case, by applying the arguments of  
 \cite[Connecting Lemma II]{yasuhara} 
and \cite[Proposition 3]{batson},
we have a new orientable surface $F$ in punc$(X \# (S^2 \times S^2))$
with boundary $K$ which satisfies the following:
\begin{enumerate}
\item $\beta_1(F) = \beta_1(N) - 1$,
\item $e(F) = e(N) + 2\varepsilon$ for some $\varepsilon = \pm1 $, and
\item $F$ represents a characteristic homology class .

\end{enumerate}
On the other hand, by the definition of Char$(Y,K)$,
there exists a disk $D$ in punc$Y$ with boundary
$K$ which represents $\eta$.
Hence by gluing $(\text{punc}(X \# (S^2 \times S^2)),F)$ 
and $(-\text{punc}Y,-D)$ along $(S^3,K)$, we
obtain the pair $(M,\Sigma)$ = $(X\# (S^2 \times S^2) \# (-Y),F \cup (-D))$
of a closed 4-manifold and an embedded,
closed, orientable surface.
Furthermore, a new orientable surface $\Sigma$ represents
a characteristic homology class $[F,\partial F]\oplus [-D, -\partial D]$ in 
$H_2(M; \mathbb{Z})$.
\begin{claim}
\label{minimal genus}
The pair $(M,\Sigma)$
satisfies the following inequality
$$
g(\Sigma) \geq \frac{|[\Sigma] \cdot [\Sigma] -
\sigma(M)|}{8}-\max(\beta^+_2(M),\beta^-_2(M))+1.
$$
\end{claim}

\def\proofname{Proof of Claim 1}
\begin{proof}
Lawson suggested that by
combining \cite[Theorem 1.2]{yasuhara}
and Furuta's theorem \cite[Theorem 1]{furuta},
we have
$$
g(\Sigma) \geq \frac{|[\Sigma] \cdot [\Sigma] -
\sigma(M)|}{8}-\max(\beta^+_2(M),\beta^-_2(M))+2
$$
if $[\Sigma] \cdot [\Sigma] \equiv \sigma(M)$ (mod 16) \cite[Theorem 36]{lawson}.
If $[\Sigma] \cdot [\Sigma] \equiv \sigma(M)+8$ (mod 16),
for any $\varepsilon'=\pm1$ by taking a connected sum with
a torus $T_{\varepsilon'}$ in $S^2 \times S^2$ with 
$[T_{\varepsilon'}]\cdot[T_{\varepsilon'}] =8\varepsilon'$ and
applying the inequality above, we have
$$
g(\Sigma) \geq \frac{|[\Sigma] \cdot [\Sigma] -
\sigma(M)|}{8}-\max(\beta^+_2(M),\beta^-_2(M))+1.
$$
\end{proof}
By Claim \ref{minimal genus},
we have the following inequality
$$
g(\Sigma) \geq \frac{|[F, \partial F] \cdot[F, \partial F] - \eta \cdot \eta - \sigma(M)|}{8}-\max(\beta^+_2(M),\beta^-_2(M))+1.
$$
Since $2g(\Sigma) = \beta_1(N)-1$, $\sigma(M)=\sigma(X)-\sigma(Y)$,
$[F, \partial F] \cdot[F, \partial F] = e(N)+ 2\varepsilon$ and $\beta_2^{\pm}(M)=\beta_2^{\pm}(X)+\beta_2^{\mp}(Y)+1$,
this inequality implies
$$
\beta_1(N) \geq \frac{|e(N)- \eta \cdot \eta - \sigma(X)+\sigma(Y)|}{4}
-2\max(\beta_2^+(X)+\beta_2^-(Y), \beta_2^-(X)+\beta_2^+(Y))+\frac{1}{2}.
$$
In the case when $\beta_1(N)$ is even, 
taking the connected sum of $N \subset X$ with the standard
embedding of $\mathbb{R}P^2 \subset S^4$ or its mirror image, 
we have a non-orientable surface $N'$ in $X$ with boundary $K$
such that $\beta_1(N')=\beta_1(N)+1$ and $e(N')= e(N)+ 2\varepsilon''$ for
any $\varepsilon'' = \pm1$.
Since $\beta_1(N')$ is odd,  we have
\begin{eqnarray*}
\beta_1(N) + 1 &\geq &\frac{|e(N)- \eta \cdot \eta - \sigma(X)+\sigma(Y)|}{4}\\
\ &\ &-2\max(\beta_2^+(X)+\beta_2^-(Y), \beta_2^-(X)+\beta_2^+(Y)) +1
\end{eqnarray*}
for a favorable value of $\varepsilon''$.
This completes the proof.
\end{proof}

In order to prove the proposition, we use the following theorem.
\begin{thm}[Yasuhara, \cite{yasuhara}]
\label{yasuhara thm}
Let $X$ be a closed $4$-manifold and
$K \subset \partial (\text{punc}X)$ a knot.
If $K$ bounds a non-orientable surface $N$ in punc$X$
that represents zero in $H_2 (\text{punc}X, \partial (\text{p-}$\\
$\text{unc}X); \mathbb{Z}_2 ) $, then
\[
\left| \sigma (K) + \sigma (X) -  \frac{e(N)}{2} \right| \leq \beta_2 ( X ) + \beta_1 ( N ).
\]
\end{thm}
\def\proofname{Proof of Proposition \ref{lower bound}}
\begin{proof}
Since we assume that $X$ is a Spin 4-manifold,
a non-orientable surface $N$ in punc$X$ represents a characteristic
homology class in 
$H_2(\text{punc}$ $X, \partial(\text{punc}X); \mathbb{Z}_2)$
if and only if $N$ represents zero-element.
Hence if $N$ satisfies the assumption of $\gamma_X^0(K)$,
then the pair $(X,N)$ satisfies both Lemma \ref{non-orientable} and
Theorem \ref{yasuhara thm}.
By combining these two inequality and deleting $e(N)$,
we can obtain the inequality of Proposition \ref{lower bound}.
\end{proof}

\section{Proof of Theorem \ref{main thm}}
We can calculate 
characteristic homology classes for certain knots
by Kirby calculus.
In this section,
we give an element of Char$(\overline{\mathbb{C}P^2},T(2k+1,2k))$ ($k>0$)
where $T(2k+1, 2k)$ denotes a $(2k+1, 2k)$-torus knot,
and prove Theorem \ref{main thm}.
The proof of Theorem \ref{main thm}
implies that for any Spin 4-manifold $X$,
the sequence $\{ \gamma^0_X(T(2k+1,2k)) \}$
always diverges.
Let $\overline{\gamma}$ be a standard generator of 
$
H_2 (\text{punc}\overline{\mathbb{C}P^2}
, \partial(\text{punc}\overline{\mathbb{C}P^2})
;\mathbb{Z})
$
such that
$\overline{\gamma} \cdot \overline{\gamma} = -1$.

\begin{prop}
$(2k+1)\overline{\gamma}$ is an element of Char$(\overline{\mathbb{C}P^2},T(2k+1,2k))$.
\end{prop}

\def\proofname{Proof}
\begin{proof}
The motion picture in Figure \ref{torus knot}
gives a disk $D$ in punc$\overline{\mathbb{C}P^2}$
with boundary $T(2k+1,2k)$ which represents $(2k+1)\overline{\gamma}$.
\end{proof}

\begin{figure}
\begin{center}
\includegraphics[clip,width=12cm]{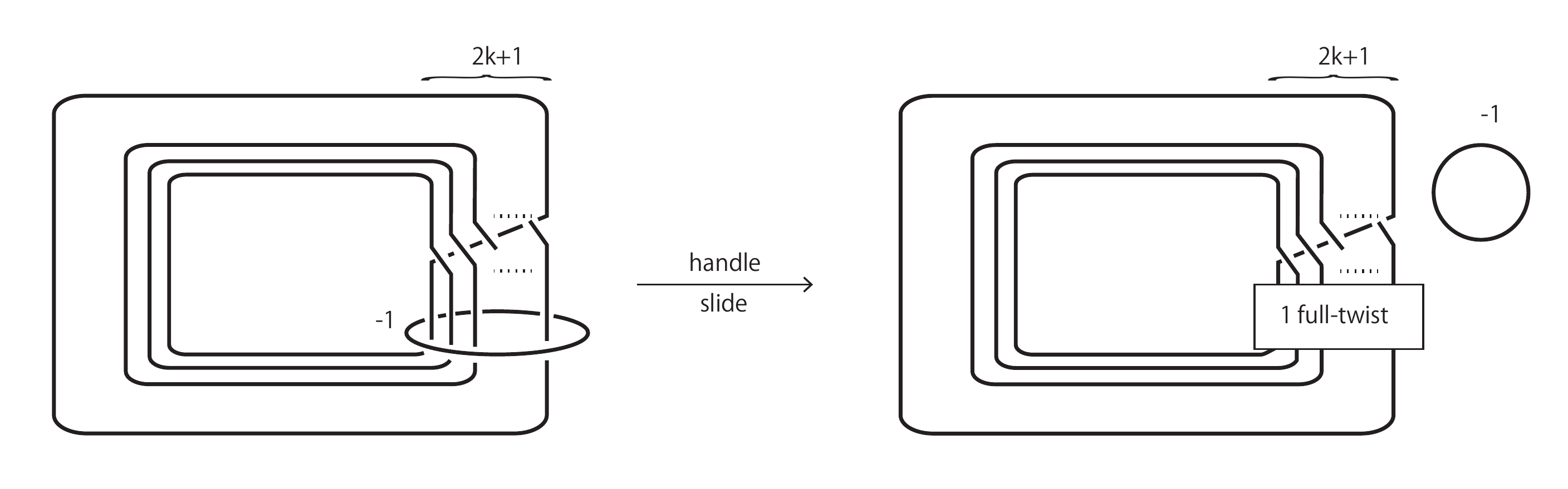}\\
\caption{}
\label{torus knot}
\end{center}
\end{figure}

\def\proofname{Proof of Theorem \ref{main thm}}
\begin{proof}
Since signatures of torus knots satisfy a recursion relation in \cite[Chapter 7]{m-k},
it is easy to check that $\sigma(T(2k+1,2k))= -2k^2$.
By applying Proposition \ref{lower bound} to the pair 
$(X,T(2k+1,2k), (2k+1)\overline{\gamma})$, we have the following inequality
$$
3\gamma^0_X(T(2k+1,2k)) \geq 
\left|2k + \frac{\sigma(X)}{2} \right| 
- 4\max(\beta^+_2(X)+1,\beta^-_2(X))-\beta_2(X),
$$
and the right side of the above inequality
infinitely diverges
when $k$ is limitlessly increased.
This completes the proof.
\end{proof}

\end{document}